\documentclass{amsart}

\usepackage[margin=1in]{geometry}
\usepackage{amsmath, amssymb}
\usepackage[all,cmtip]{xy}
\usepackage{color,graphicx}
\usepackage{amsthm}
\usepackage{amsfonts}
\usepackage{amscd}
\usepackage{comment}
\usepackage{mathrsfs}
\usepackage[toc,page]{appendix}
\usepackage[utf8]{inputenc}
\usepackage{tikz, subfigure}
\usepackage[mathscr]{euscript}
\usepackage{enumitem,kantlipsum}

\usepackage[bookmarks=true, bookmarksopen=true,%
bookmarksdepth=3,bookmarksopenlevel=2,%
colorlinks=true,%
linkcolor=blue,%
citecolor=blue,%
filecolor=blue,%
menucolor=blue,%
urlcolor=blue]{hyperref}
\usepackage{cleveref}

\def\EE{\mathbb{E}}

\def\HH{\mathbb{H}}

\def\calG{\mathcal{G}}

\newcommand\frS{\mathfrak{S}}

\newcommand\aff{\textup{aff}}

\newcommand{\Tr}{\textup{Tr}}

\newcommand\Map{\textup{Map}}

\newcommand\nc{\newcommand}
\nc\on{\operatorname}
\nc\ol{\overline}
\nc\ul{\underline}
\nc\us[1]{\underline{\smash{#1}}}

\nc\oo{\infty}

\nc\Cone{\mathit{Cone}}
\nc\ssupp{\mathit{ss}}
\nc\risom{\stackrel{\sim}{\to}}
\nc\Sh{\textup{Sh}}
\nc\un{\diamondsuit}
\nc\orient{\mathit{or}}
\nc\sing{\mathit{sing}}
\nc\MF{\on{MF}}
\nc\inthom{\mathit{Hom}}

\newcommand{\Z}{\mathbb{Z}}

\newcommand{\colim}{\textup{colim}}

\newtheorem{thm}[equation]{Theorem}
\newtheorem{prop}[equation]{Proposition}
\newtheorem{lem}[equation]{Lemma}

\theoremstyle{definition}
\newtheorem{defn}[equation]{Definition}

\newtheorem{eg}[equation]{Example}

\newtheorem{rmk}[equation]{Remark}
\newtheorem{conj}[equation]{Conjecture}

\numberwithin{equation}{section}

\begin{document}

\nc{\Conn}{\mathrm{Conn}}

\nc{\fg}{\mathfrak g}
\nc{\fh}{\mathfrak h}

\nc{\cN}{\mathcal N}


\title[A Colimit of Traces of Reflection Groups]{A Colimit of Traces of Reflection Groups}

\author{Penghui Li}

\address{Institute of Science and Technology Austria}
\email{pli@ist.ac.at}

\begin{abstract}
Li-Nadler proposed a conjecture about traces of Hecke categories, which implies the  semistable part of the Betti Geometric Langlands Conjecture of Ben-Zvi-Nadler in genus 1. We prove a Weyl group analogue of this conjecture. Our theorem holds in the natural generality of reflection groups in Euclidean or hyperbolic space.
\end{abstract}

\maketitle


\section{Introduction}

Let $W$ be a reflection group in Euclidean or hyperbolic space. For $I$ a facet, denote by $W_I$ the subgroup fixing $I$.  For $C$ a chamber, denote by $\mathscr{F}_C$ be the category (or poset) of faces in $\overline{C}$ ($:=$ the closure of $C$). We view $W$ (with discrete topology) as an algebra object in $\mathscr{S}:=$ the $\infty$-category of topological spaces, and denote its trace by $\Tr(W) \in \mathscr{S}$. Our main theorem is:
\begin{thm} 
	\label{main}
	The natural map 
	$$\xymatrix{ \textup{colim}_{I \in \mathscr{F}_C^{op}} \Tr(W_I)  \ar@{^{(}->}[r]  &  \Tr(W)          }$$
	is fully-faithful.
\end{thm}
The proof uses Lurie's $\infty$-categorical Seifert-Van Kampen theorem and basic properties of reflection groups. 

When $W=W_{\aff}$ the affine Weyl group of a simply-connected reductive group $G$, Theorem~\ref{main} confirms a Weyl group analogue of the following conjecture in \cite{LN}. Let $G$ be a simply-connected reductive algebraic group, $LG$ the loop group of $G$. Let $C$  be an affine alcove. For each face $I$ of $C$, denote by $G_I$ the Levi of the parahoric subgroup of $LG$ corresponding to $I$. Let $\mathscr{H}_I$ be the Hecke category of $G_I$, and $\mathscr{H}_{\aff}$ be the affine Hecke category. 
\begin{conj}[{\cite[Claim 1.12]{LN}}]
	\label{conjecture}
	The natural map of $\infty$-categories 	
	$$\xymatrix{ \textup{colim}_{I \in \mathscr{F}_C^{op}} \Tr(\mathscr{H}_I)  \ar[r]  &  \Tr(\mathscr{H}_{\aff})          }$$
	is fully-faithful.
\end{conj}

This conjecture comes from the consideration of Geometric Langlands. Roughly speaking, the Betti Geometric Langlands Conjecture \cite{BZN16} predicts the equivalence of two ($\infty$)-categories: the automorphic category $\mathscr{A}_g$ and the spectral category $\mathscr{B}_g$. As explained in \cite{LN}, The above conjecture implies that for genus $g=1$, one can embed the semistable automorphic category $\mathscr{A}_1^{ss} \subset \mathscr{A}_1$ fully-faithfully into $\mathscr{B}_1$, and hence implies this part of Geometric Langlands. Note that $W_I$ is the Weyl group of $G_I$,  Weyl groups are specializations of Hecke algebra, and Hecke algebras are decategorifications of Hecke categories. Hence Theorem~\ref{main} confirms an easier analogue of Conjecture~\ref{conjecture}.

 \begin{rmk} 
 	\begin{enumerate}[leftmargin=*]
 	\item 	
 	For a topological group $G$ acting on a topological space $X$, we denote by $X/G$ the topological space $X \times_G EG$, where $EG$ is a contractible space with free $G$ action. It is not hard to see that $\Tr(G) \simeq G/G$, for the adjoint action of $G$ (Proposition~\ref{trace of a group}). Denote $\bullet$ a single point.
 	 It is known that $\textup{colim}_{I \in \mathscr{F}_C^{op}} \bullet/W_I \simeq \bullet/W$ (see e.g. \cite{Li2}).  This equivalence sits inside Theorem~\ref{main} via the commutative diagram
	$$\xymatrix{\textup{colim}_{I \in \mathscr{F}_C^{op}} \bullet/W_I  \ar[r]^-{\sim}  \ar@{^{(}->}[d] &  \bullet/W  \ar@{^{(}->}[d]  \\
		 \textup{colim}_{I \in \mathscr{F}_C^{op}} W_I/W_I  \ar@{^{(}->}[r]  &  W/W         }$$
 	 where the vertical maps take $\bullet$ to $1$. A similar statement of the top equivalence for the Bruhat-Tits building was used to prove that the representation 
 	 category of a $p$-adic group has global dimension $\leq \dim(C)$ (see e.g. Bernstein's lectures on Representation of $p$-adic groups). It may be interesting to see the meaning of bottom arrow in $p$-adic representation theory.
 	 \item We can get the map in Theorem~\ref{main}  by applying $\Map(S^1,-)$ to the equivalence $\colim \bullet/W_I \simeq \bullet/W $. However, this resulting map is no longer an equivalence in general. This reflects the fact that $\Map(S^1,-) $ does not preserve colimit. I.e, the loop space are not calculated locally. To see a concrete example when the surjectivity fails: take $W=W_{\aff}$, then $\pi_0(LHS)$ is finite but $\pi_0(RHS) \supset \{\text{dominant coweights}\}$, which is infinite.
 	\end{enumerate}

\end{rmk} 

 \begin{eg}
 	We gives some examples of Theorem~\ref{main}. Denote by $\mathfrak{S}_n$ the symmetric group on $n$ letters, 
 	\begin{enumerate}[leftmargin=*]
 		\item $W$ is the Weyl group of a reductive algebraic group $G$. Then $\mathscr{F}_C^{op}$ has a final object $O$ the origin, and $W_O=W$. Hence Theorem holds trivially since the LHS is also $\Tr(W)$.
 		\item $W$ is the affine Weyl group of $SL_2$. $\mathscr{F}_C^{op}$ is the category $\bullet \leftarrow \bullet \to \bullet$. $\frS_2/\frS_2 \simeq \bullet/\mathfrak{S}_2 \coprod \bullet/\mathfrak{S}_2$ \\
 		$LHS =  \colim  
 		\xymatrixcolsep{0.05in} 
 		\xymatrix{  & \bullet \ar[dl] \ar[dr] &  \\
 		 \frS_2/\frS_2    &  &  \frS_2/\frS_2  
 	}      \simeq 
  \colim  
 \xymatrixcolsep{0.05in}
 \xymatrix{  & \bullet \ar[dl] \ar[dr] &  \\
 	\bullet/\frS_2    &  &  \bullet/\frS_2 }
 \coprod \;\;\;
   \colim  
\xymatrixcolsep{0.05in}
 \xymatrix{ 	 
 	 &  &  \\
 	\bullet/\frS_2    &  &  \bullet/\frS_2 }
    \\
\simeq \bullet/W  \coprod  \bullet/\mathfrak{S}_2   \coprod \bullet/\mathfrak{S}_2      \to W/W$

The image is the full subgroupoid consist of objects $1,s_1,s_0$, for $s_1,s_0$ two simple reflections in $W$. The map being fully-faithful reflects the fact that $s_1$ and $s_0$ are not conjugate in $W$, and the centralizer of each is $\mathfrak{S}_2$.
  \item $W$ is the affine Weyl group of $SL_3$. Note that $\frS_3/\frS_3 \simeq \bullet/\frS_3 \coprod \bullet/\frS_2 \coprod \bullet/(\mathbb{Z}/3)$. \\
$LHS= \colim  \xymatrixrowsep{0.05in}
  \xymatrixcolsep{-0.15in}
  \xymatrix{
  	&    &  \frS_3/\frS_3 \ar@{<-}[dddl] \ar@{<-}[dddr] \ar@{<-}[dddd] &    &        \\
  	&    &    &   &     \\
  	&    &    &   &     \\
  	& \frS_2/\frS_2
  	\ar@{<-}[dr]  &   & \frS_2/\frS_2 \ar@{<-}[dl] &     \\
  	&    &  \bullet &    &        \\
  	&    &    &    &        \\
\frS_3/\frS_3   \ar@{<-}[uuur] \ar@{<-}[rr] \ar@{<-}[uurr]   &    &  \frS_2/\frS_2  \ar@{<-}[uu]
  	&      & \frS_3/\frS_3  \ar@{<-}[uuul] \ar@{<-}[ll] \ar@{<-}[uull]
}   \\
\simeq  \colim \xymatrixrowsep{0.05in}
\xymatrixcolsep{-0.10in}
\xymatrix{
	&    &  \bullet/\frS_3 \ar@{<-}[dddl] \ar@{<-}[dddr] \ar@{<-}[dddd] &    &        \\
	&    &    &   &     \\
	&    &    &   &     \\
	& \bullet/\frS_2
	\ar@{<-}[dr]  &   & \bullet/\frS_2 \ar@{<-}[dl] &     \\
	&    &  \bullet &    &        \\
	&    &    &    &        \\
	\bullet/\frS_3   \ar@{<-}[uuur] \ar@{<-}[rr] \ar@{<-}[uurr]   &    &  \bullet/\frS_2  \ar@{<-}[uu]
	&      & \bullet/\frS_3  \ar@{<-}[uuul] \ar@{<-}[ll] \ar@{<-}[uull]  }
\coprod  \;\;
\colim \xymatrixrowsep{0.05in}
\xymatrixcolsep{-0.10in}
\xymatrix{
	&    &  \bullet/\frS_2 \ar@{<-}[dddl] \ar@{<-}[dddr]  &    &        \\
	&    &    &   &     \\
	&    &    &   &     \\
	& \bullet/\frS_2
	&   & \bullet/\frS_2 &     \\
	&    &   &    &        \\
	&    &    &    &        \\
	\bullet/\frS_2   \ar@{<-}[uuur] \ar@{<-}[rr]    &    &  \bullet/\frS_2 
	&      & \bullet/\frS_2  \ar@{<-}[uuul] \ar@{<-}[ll]   } 
 \coprod  \;\; \colim
 \xymatrixcolsep{-0.05in}
 \xymatrix{
 	&    &  \bullet/(\Z/3)  &    &        \\
 	&    &    &   &     \\
 	&    &    &   &     \\
 	& 	&   &  &     \\
 	&    &   &    &        \\
 	&    &    &    &        \\
 	\bullet/(\Z/3)    &    &  	&      & \bullet/(\Z/3)    } 
 \\
 \simeq  \bullet/W \; \coprod \; (\bullet/\frS_2 \times \partial C) \; \coprod  \; (\bullet/(\mathbb{Z}/3))^{\coprod 3} . $ 
 
 The second factor $\bullet/\frS_2 \times \partial C$ can be identified as the full subgroupoid in $W/W$ consists of reflections: let $s_I \in W$ be the reflection correpsonding to a face $I$ of $C$. The centralizer $C_W(s_I) \simeq {<s_I>} \times X_*(Z(G_I)) \simeq \frS_2 \times \mathbb{Z}.$ Hence the subgroupoid at $s_I$ is equivalent to $\bullet/\frS_2 \times \bullet/\mathbb{Z} \simeq \bullet/\frS_2 \times S^1 \simeq \bullet/\frS_2 \times \partial C.$ Also note that all reflections are conjugate in this case.
 \item Let $W$ be the triangle group $(2,3,\infty)$. It is a reflection group in hyperbolic plane. \\
 $LHS = \colim \xymatrix{ \frS_2/\frS_2 \ar[d] &  \bullet \ar[l] \ar[r] \ar[d] &  \frS_2/ \frS_2 \ar[d] \\
 	 (\frS_2 \times \frS_2)/(\frS_2 \times \frS_2) & \frS_2/\frS_2 \ar[l] \ar[r]  &  \frS_3/\frS_3 } \\
   = \colim \{ \xymatrix{
 	 (\frS_2 \times \frS_2)/(\frS_2 \times \frS_2) & \frS_2/\frS_2 \ar[l] \ar[r]  &  \frS_3/\frS_3 }  \} \\
  = \bullet/W \coprod (\bullet/(\frS_2 \times \frS_2)) ^{\coprod 3} \coprod \bullet/(\mathbb{Z}/3) $ \\
   We see that in this case, $C_{W_I}(w) = C_W(w)$, for any $w \in W_I$.

	\end{enumerate}	
 \end{eg}

\section{Preliminaries}
\subsection{Discrete groups generated by reflections}
References for this section are \cite[V]{Bour} and \cite{Vin88}. 
Let $X$ be an Eucliean space $\EE^n$ or hyperbolic space $\HH^n$. Let $\mathfrak{H}$ be a collection of hyperplanes in $X$. Let $W$ be the group generated by the orthogonal reflection along the hyperplanes $H \in \mathfrak{H}$. Assume that: 
\begin{enumerate}
\item For any $w \in W$ and $H \in \mathfrak{H}$, we have $w(H) \in \mathfrak{H}$.
\item  $W$ provided with discrete topology, acts properly on $X$.
\end{enumerate}
Given two points $x$ and $y$ of $X$, denote by $R\{x,y\}$ the equivalence relation: \\
For any hyperplane $H \in \mathfrak{H}$, either $x \in H$ and $y \in H$ or $x$ and $y$ are strictly on the same side of $H$. 

\begin{defn}
	\label{facet}
	\begin{enumerate}
		\item A \textit{facet} of $X$ is an equivalence class of the equivalence relation defined above.
		\item For two facet $I,J$, denote $I \leq J$ if $I \subset \overline{J}$. Then $\leq $ defines an partial order on the set of facets.
		\item A \textit{chamber} of $X$ is a facet $C$ that is maximal along this partial order.
		\item For any facet $J$, denote $\mathscr{F}_J$ be the category corresponding to the poset $\{ I |  I \leq J\}.$
		\item The \textit{star} of $I$ is $X_I:=\bigcup_{\{J | I \leq J\}} J \subset X$; and  $W_I:=\{w \in W: w|_I =id\}$.
	\end{enumerate}
\end{defn}

\begin{prop}
	\label{fundamental domain}
	\begin{enumerate}
		\item A facet is a polytope.
		\item For any chamber $C$, the closure $\overline{C}$ of $C$ is a fundamental domain for the action of $W$ on $X$, i.e., every orbit of $W$ in $X$ meets $\overline{C}$ in exactly one point. 
		\item For $I$ a facet, the group $W_I$ is generated by the reflections fixing $I$. 
		\item $W_I$ acts on $X_I$ with fundamental domain $X_I \cap \overline{C}$.
	\end{enumerate}
\end{prop}

\begin{proof} (1) by defintion since each facet is a intersection of hyperplanes and half spaces. (2) See \cite[V.3.3 Theorem 2]{Bour}. (3) See \cite[V.3.3 Prop 2]{Bour}. (4) Let $J \geq I$ be a facet in $X_I$, then $w(J) \geq w(I) =I$, hence $w(J) \subset X_I$, and therefore $W_J$ acts on $\overline{C}$. For the second statement, note that $W_I$ is also a reflection group, hence by (2), $X_I \cap \overline{C}$ is a fundamental domain.

\end{proof}

 \subsection{Traces of algebras} Let $\mathscr{C}$ be a symmetric monoidal $\infty$-category such that all colimit exist. Let $A \in \text{Alg}(\mathscr{C})$ be an algebra object in $\mathscr{C}$. The \textit{trace (or Hochschild homology)}  of $A$ is by definition $\Tr(A):= A \otimes_{A \otimes A^{op}} A \in \mathscr{C}$. We view $\mathscr{S}$ as a symmetric monoidal $\infty$-category, where $\otimes$ is given by the Cartesian product.
 \begin{prop}
 	\label{trace of a group}
 	Let $G$ be a topological group. View $G \in \textup{Alg}(\mathscr{S})$, then $\Tr(G) \simeq G/G:= G \times_G EG$, where $G$ acts on $G$ by conjugation. 
\end{prop}

\begin{proof}
	We have an isomorphism  $G \simeq G \otimes G^{op} \otimes_G \bullet$ as $G \otimes G^{op}$ modules, where $G$ maps to $G \otimes G^{op}$ diagonally. Then
	$G \otimes_{G \otimes G^{op}} G \simeq G\otimes_{G \otimes G^{op}} (G \otimes G^{op} \otimes_G \bullet) \simeq G \otimes_G \bullet \simeq G/G$, where the action of $G$ on $G$ is the conjugation.
\end{proof}	

\subsection{Topological groupoid and open descent}
We denote a \textit{topological groupoid} $\mathcal{G}$ to be the data consist of a discrete group $G$ acting properly discontinuously on a topogical space $Y$. And we use the notation $\mathcal{G}=[Y/G]$. Let $\mathcal{G}'=[Y'/G']$ be another topological groupoid. A $morphism$  $F :\mathcal{G} \to \mathcal{G}'$ consists of the data $(f, \varphi)$, where $f:Y \to Y'$ continous maps and $\varphi : G \to G'$ injective homomorphisms, such that $ f(a \cdot y)=  \varphi(a) \cdot f(z), $ for all $a \in G, y \in Y$.  We denote \textup{TopGrpd} the category of topological groupoids. A morphism $F$ is an \textit{open embedding} if the induced map $ Y \times_{G} G' \to  Y'$  is an open embedding. We denote $\us{Y}$ the underline set of $Y$, and by $ \us{\mathcal{G}} := \us{Y}/G \in \mathscr{S}$ the underline $\infty$-groupoid of points (recall that $G$ is assumed to be discrete). And define $\mathcal{G}_h:= Y/G \in \mathscr{S}$ the homotopy type of $\mathcal{G}$.

\begin{prop}
	\label{open embedding}
	Let $F : \mathcal{G} \to \mathcal{G}'$, assume that $f$ is open embeddings and the induced map $\us{F}:\us{\mathcal{G}} \to \us{\mathcal{G}}'$ is fully-faithful. Then $F$ is an open embedding.
\end{prop}	
\begin{proof} The base changed map $Y \times_G G' \to Y'$ is a local homeomorphism since $G$ and $G'$ are discrete. The underline set map $\us{Y} \times_G G' \to \us{Y}'$ is fully-faithful (i.e injective), because fully-faithful map between groupoid are stable under base change. These imply the map $Y \times_G G' \to Y'$ is an open embedding.
	
\end{proof}	

For $\mathscr{I}$ a category, we denote $\mathscr{I}^\triangleright$ the category by adding one final object $*$ to $\mathscr{I}$. We say a functor $K: \mathscr{I}^\triangleright \to \mathscr{T}$ is a \textit{colimit diagram} if the induced map $\colim K|_\mathscr{I} \to K(*)$ is an isomorphism.

\begin{prop}[$\infty$-categorical Seifert-van Kampen theorem for topological groupoids] 
	\label{colimit} 
Let $K: \mathscr{I}^\triangleright \to \textup{TopGrpd} $ be a functor, assume that all arrows in $\mathscr{I}^\triangleright$ go to open embeddings, and the induced functor $\us{K}: \mathscr{I}^\triangleright \to \mathscr{S}$ is a colimit diagram. Then the induced functor $K_h:\mathscr{I}^\triangleright \to \mathscr{S}$ is a colimit diagram.
\end{prop}	

\begin{proof}
By base change, one can assume that $K$ takes value in $\textup{Top}$ the category of topological spaces. Then this is the $\infty$-categorical Seifert-van Kampen theorem \cite[Theorem A.3.1]{Lur12}. Note that the condition $(*)$ \textit{loc. cit.} is equivalent to the condition on $\us{K}$.
\end{proof}	

\begin{rmk}[Topological groupoids as topological stacks] Denote \textup{Top} the category of topological space with continuous map. One can define a topogical stack as a functor $X: \textup{Top} \to \mathscr{S}$, satisfying certain descent and representibility conditions. Then Yoneda embedding gives $\iota: \text{Top} \hookrightarrow \text{TopStack}$. One can define embedding $\iota':\text{TopGrpd} \to \text{TopStack}$, via $[Y/G] \to \colim_{\bullet \in \Delta^{op}} \iota(G^{\times \bullet} \times Y).$ In this case, $\us{\calG} = \iota'(\calG)(*)$ and $\calG_h$ is also the homotopy type of $\iota'(\calG)$. Proposition~\ref{colimit} is most naturally presented in the context of topological stacks (with local homeomorphisms), but we shall not use this generality.
\end{rmk}	

\section{Proof of main theorem}

	For any $w \in W$, let $X^w$ be the fixed locus of $w$. For $I$ a facet, put $X_I^w := X_I  \cap X^w $.  Let $W^f := \{w \in W : w(I)=I, \text{for some facet } I\}$. We define a funtor $K: \mathscr{F}^{op,\triangleright}_{C} \to \textup{TopGrpd},$ by $I \mapsto [(\coprod_{ w \in W_I} X^w_I)/W_I]$, and $* \mapsto [(\coprod_{w \in W^f} X^w)/W]$. 
		
\begin{lem}
	\label{groupoid}
	$\us{K(I)}= \coprod_{\{J| C \geq J \geq I\}} (\us{J} \times W_J)/W_J$, and $\us{K(*)}= \coprod_{ \{ J|C \geq J\}} (\us{J} \times W_J)/W_J.$
\end{lem}		

\begin{proof}
	As sets, we have ${X^w_I} = \coprod_{\{J| J \geq I, w(J)=J\}} J,$ and ${\coprod_{w \in W_I} X^w_I} =  \coprod_{w \in W_I} \coprod_{\{J| J \geq I, w(J)=J\}} J  \\
	= \coprod_{\{(J,w)|J \geq I, w(J)=J, w \in W_I \} } {J \times \{w\}} = \coprod_{\{(J,w)|J \geq I, w \in W_J\}} {J \times \{w\}} =\coprod_{\{J| J \geq I\}} J \times W_J$. Hence  $\us{K(I)}=(\coprod_{\{J| J \geq I\}} \us{J} \times W_J)/W_I = \coprod_{\{J| C \geq J \geq I\}} (\us{J} \times W_J)/W_J,$ where the last equality is by Proposition~\ref{fundamental domain} (2). The second statement follows from similar argument.
\end{proof}	

\begin{lem}
	\label{functor K}
	\begin{enumerate}
		\item For any $I' \to I$ in $\mathscr{F}^{op,\triangleright}_{C}$, $\us{K(I')} \to \us{K(I)}$ is fully-faithful.
		\item $\us{K}$ is a colimit diagram.
	\end{enumerate}	
\end{lem}
 \begin{proof}
 (1)  One check that under the identification in Lemma~\ref{groupoid}, the map $\us{K(I')} \to \us{K(I)}$ is induced by the inclusion of indexing sets ${\{J| C \geq J \geq I'\}} \to {\{J| C \geq J \geq I\}}$. \\
 (2) For any $J  \leq C$,  define $\us{K}_J : \mathscr{F}^{op,\triangleright}_{C} \to \mathscr{S}$ by 
 \begin{equation}
\us{K}_J(I):=
 \begin{cases}
 (\us{J} \times W_J)/W_J, & \text{if}\ I \leq J \\
 \emptyset, & \text{otherwise}
 \end{cases}
 \end{equation}
 We see that $\colim_{\mathscr{F}^{op}_{C} } \us{K}_J \simeq | \mathscr{F}^{op}_J | \times ((\us{J} \times W_J)/W_J) \simeq (\us{J} \times W_J)/W_J \simeq \us{K}_J(*) $. The second equivalence follows from the fact that the geometric realization $ |\mathscr{F}^{op}_J| \simeq \overline{J}$  is contractible. Hence $\us{K}_J$ is a colimit diagram, and $\us{K} \simeq \coprod_{J \leq C}\us{K}_J$ is also a colimit diagram.
 \end{proof}

\begin{proof}[Proof of Theorem~\ref{main}]	By Propsition~\ref{trace of a group}, it is equivalent to show the natural map $\colim_{I \in \mathscr{F}^{op}}  W_I/W_I \to W/W$ is fully-faithful.	
We claim the functor $K$ satisfies the assumption of Proposition~\ref{colimit}.	We first show that all arrows in $ \mathscr{F}^{op,\triangleright}_{C}$ goes to open embeddings. For any $I' \geq I $, the natural map $\coprod_{w \in W_{I'}}X^w_{I'} \to \coprod_{w \in W_I}X^w_{I}$ is an open embedding. Hence by Proposition~\ref{open embedding}, and Lemma~\ref{functor K} (1), $K(I') \to K(I)$ is an open embedding. And $\us{K}$ is a colimit diagram by Lemma~\ref{functor K} (2). Hence we conclude that $K_h$ is a colimit diagram by Proposition~\ref{colimit}. Now we have a commutative diagram in $\mathscr{S}$:
$$\xymatrix{
	\colim_{I \in \mathscr{F}^{op}_{C} } (\coprod_{w \in W_I} X_I^w)/ W_I   \ar[r]^-{\sim} \ar[d]^{\sim} & (\coprod_{w \in W^f} X^w) /W \ar[d]^{\sim} \\
	 \colim_{I \in \mathscr{F}^{op}_{C} } W_I / W_I \ar[r]^p \ar[rd]^q & W^f /W \ar@{^{(}->}[d]^i\\
	  &    W / W
}$$
In the top square, the top horizontal arrow is an equivalence from the definition of  $K_h$ being a colimit diagram. The two vertical arrows are given by $X^w_I (resp. \; X^w) \mapsto \{w\}$, hence they are equivalences since $X^w_I$ and $X^w$ are contractible. We conclude that $p$ is an equivalence. Now $i$ is fully-faithful by definition, hence $q$ is fully-faithfully.
\end{proof}

\section{Acknowledgements}
We would like to thank Roman Bezrukavnikov, Anna Felikson, Dragos Fratila, Sam Gunningham, Quoc Ho, Jacob Lurie, Nitu Kitchloo, David Nadler and Roman Travkin for helpful conversations. 
The author is grateful for the support of Prof. Tamas Hausel and the Advanced Grant “Arithmetic and Physics of Higgs moduli spaces” No. 320593 of the European Research Council.

\bibliographystyle{alpha}
\bibliography{paper}

\end{document}